\documentclass[11pt,leqno]{amsart}

\usepackage[top=2.5 cm,bottom=2 cm,left=2.5 cm,right=2 cm]{geometry}
\usepackage[utf8]{inputenc}

\usepackage{bbm}
\usepackage{esint}
\usepackage{graphicx}
\usepackage[hidelinks]{hyperref} 
\usepackage{color}
\usepackage{amsfonts}
\usepackage{psfrag}
\newcounter{stepnb}

\usepackage{xfrac}

\usepackage{tikz}
\usetikzlibrary{shapes.geometric,calc,decorations.markings,math}

\tikzstyle{nodo}=[circle,draw,fill,inner sep=0pt,minimum size=%
1.5mm]
\tikzstyle{infinito}=[circle,inner sep=0pt,minimum size=0mm]

\newtheorem{theorem}{Theorem}[section]
\newtheorem{lemma}[theorem]{Lemma}

\newtheorem{remark}[theorem]{Remark}

\newtheorem{definition}[theorem]{Definition}
\numberwithin{equation}{section}

\newcommand{\R}{\mathbb{R}}

\DeclareMathOperator*{\essup}{ess\,sup}

\newcommand{\ee}{\varepsilon}

\newcommand{\be}{\begin{equation}}
\newcommand{\eq}{\end{equation}}

\renewcommand{\div}{{\rm div}\,}

\newcommand{\weaks}{\stackrel{*}{\rightharpoonup}}
\newcommand{\loc}{\mathrm{loc}}

\begin{document}
\title[Nonlocal conservation laws with $BV$ kernels]{On multidimensional nonlocal conservation laws \\ with $BV$ kernels
}

\author[M.~Colombo]{Maria Colombo}
\address{M.C. EPFL B, Station 8, CH-1015 Lausanne, Switzerland.}
\email{maria.colombo@epfl.ch}
\author[G.~Crippa]{Gianluca Crippa}
\address{G.C. Departement Mathematik und Informatik,
Universit\"at Basel, Spiegelgasse 1, CH-4051 Basel, Switzerland.}
\email{gianluca.crippa@unibas.ch}
\author[L.~V.~Spinolo]{Laura V.~Spinolo}
\address{L.V.S. IMATI-CNR, via Ferrata 5, I-27100 Pavia, Italy.}
\email{spinolo@imati.cnr.it}
\maketitle
{
\rightskip .85 cm
\leftskip .85 cm
\parindent 0 pt
\begin{footnotesize}
We establish local-in-time existence and uniqueness results for nonlocal conservation laws in several space dimensions under weak (that is, Sobolev or $BV$) differentiability assumptions on the convolution kernel. In contrast to the case of a smooth kernel, in general the solution experiences finite-time  blow up. We provide an explicit example showing that solutions corresponding to different smooth approximations of the convolution kernel in general converge to different measures after the blow-up time. This rules out a fairly natural strategy for extending the notion of solution of the nonlocal conservation law after the blow-up time. 

\medskip

\noindent
{\sc Keywords:} nonlocal conservation laws in several space dimensions, models for pedestrian traffic, well-posedness, $BV$ kernels, nonuniqueness, lack of selection


\medskip\noindent
{\sc MSC (2020):  35L65, 35A01, 35R06}

\end{footnotesize}
}

\section{Introduction}
We consider the following Cauchy problem for nonlocal conservation laws in several space dimensions, 
\begin{equation}
\label{e:nonlocal}
     \left\{
     \begin{array}{ll}
    \partial_t u + \mathrm{div} [u V (t, x, u \ast \eta)] =0 , \\
     u(0, \cdot) = u_0 , \\
    \end{array}
     \right.
\end{equation}
where $u:  \R_+ \times \R^d \to \R$ is the unknown, $V: \R_+ \times \R^d \times \R^N \to \R^d$ is a Lipschitz continuous function, and $\eta \in L^1_{\loc} (\R^d; \R^N)$ is a convolution kernel. We denote by $\mathrm{div}$ the divergence computed with respect to the space variable only, whereas the symbol $\ast$ stands for the convolution with respect to the space variable only, that is 
$$
    u \ast \eta (t, x) : = \int_{\R^d} u(t, x-y) \eta(y) dy.
$$
In this work we establish local-in-time existence and uniqueness results for~\eqref{e:nonlocal} under fairly weak differentiability assumption on $\eta$, in particular we focus on the case of Sobolev and $BV$ (bounded total variation) regularity. We also consider an explicit example where the solution $u$ experiences finite time blow up and show that, in general, solutions corresponding to different smooth approximations of $\eta$ converge to different measures after the blow-up time. 

In the last few years, the analysis of nonlocal conservation laws has attracted quite some attention, owing to the wide range of applications spanning from sedimentation~\cite{Sedimentation} to supply chains~\cite{CHM}, and pedestrian~\cite{CGLM,PiccoliTosin} and vehicular~\cite{Chiarello} traffic models. In the vehicular traffic framework, corresponding to the one-dimensional case $x \in \R$, several recent works have focused on the case of \emph{anisotropic} convolutions kernels, that is, $\eta$ is supported on the negative real axis~$\R_-$. This aims at modeling the fact that drivers tune their velocity based on the downstream traffic density only. We refer in particular to the original paper~\cite{BlandinGoatin}, the recent overview~\cite{Chiarello} and the references therein for a non exaustive list of recent contributions establishing global-in-time existence and uniqueness results, maximum principle and many other properties. Anisotropic convolution kernels typically have at least one jump discontinuity at the origin, and as such do not satisfy the smoothness properties required in previous works on nonlocal conservation laws, see for instance~\cite{GianMag}. In the one-dimensional case, a systematic treatment of nonlocal conservation laws with $BV$ kernels 
is provided in~\cite[Proposition 3.2]{GoatinRossi} and~\cite{CocliteDeNittiKeimerPflug}. At the end of~\cite{CocliteDeNittiKeimerPflug} the authors quote as an open question the extension of their analysis to the multidimensional case. In the present work, we first of all provide an answer to this question. Note that well-posedness results for the nonlocal Cauchy problem~\eqref{e:nonlocal} in the multidimensional case were already well known, see for instance in~\cite{ACG,GianMag,KPS}, but to the best of our knowledge all previous results require much stronger regularity on $\eta$ than the one we consider in the following.

A primary motivation for our analysis comes from pedestrian traffic or crowd dynamics models: as an example, in~\cite{PiccoliTosin} the authors model 
pedestrian flows with equations similar to~\eqref{e:nonlocal} under the choice 
$$
    V (t, x, u\ast \eta)= V_d (t, x) + [u \ast \eta] (t, x),
$$ 
where $u$ represents the crowd density and $V_d$ the desired velocity, that is the velocity that a pedestrian at the point $x$ at time $t$ would have if there were no crowd. The term $ u \ast  \eta$ represents a correction term modeling the fact that pedestrians want to avoid very crowded areas. In~\cite[formula (22)]{PiccoliTosin} the authors consider the discontinuous convolution kernel 
$$
    \eta (x) : = \frac{x}{|B_r(0)|} \mathbbm{1}_{B_r(0)} (x),
$$
where $B_r(0)$ denotes the ball of radius $r$ centered at the origin, and $\mathbbm{1}$ the characteristic function. See also~\cite{CGLM,CHM,CDFFL,ColomboRossi,PiccoliRossi} for related models. We now state our main existence result. 
\begin{theorem} \label{t:ex}
Let $V: \R_+ \times \R^d \times \R^N \to \R^d$  be an $L$-Lipschitz continuous function. Then the following holds. 
\begin{itemize}
\item[(i)] If $\eta \in BV (\R^d;\R^N)$ and $u_0 \in L^\infty (\R^d)$, then there is $T^\ast>0$, only depending on $L$, $d$, $N$,  $|D \eta| (\R^d)$ and on $\| u_0 \|_{L^\infty}$, 
and a distributional solution  $u \in L^\infty_\loc ([0, T^\ast[ ; L^\infty(\R^d))$
of the Cauchy problem~\eqref{e:nonlocal}.  
\item[(ii)] Fix $p>1$ and let $q = \sfrac{p}{(p-1)}$. If $\eta \in W^{1, p} (\R^d;\R^N)$ and $u_0 \in L^q (\R^d)$, then there is $T^\ast_q>0$, only depending on $L$, $d$, $N$, $\| \nabla \eta \|_{L^p (\R^d; \R^N \times \R^d)}$ and on $\| u_0 \|_{L^q (\R^d)}$, 
and a distributional solution $u \in L^\infty_\loc([0, T^\ast_q[ ; L^q(\R^d))$ of the Cauchy problem~\eqref{e:nonlocal}. 
\end{itemize}
\end{theorem}
The proof of Theorem~\ref{t:ex} relies on an approximation argument combined with a control on the growth of $u$ 
obtained through the classical methods of characteristics. By the above result, observe that either 
$$
      \essup_{t \in [0, T^\ast[}  \| u (t, \cdot) \|_{L^\infty}= + \infty, \qquad \text{respectively} \qquad 
       \essup_{t \in [0, T^\ast_q[ }  \| u (t, \cdot) \|_{L^q}= + \infty,
$$  
or $T^\ast$, respectively $T^\ast_q$, is not maximal and one can construct a solution defined on a larger time interval.
By inspecting the proof one obtains the following lower bounds on $T^\ast$ and $T^\ast_q$ respectively  
$$
      T^\ast \ge  \sum_{j=1}^\infty \frac{\ln 2}{K_1 + 2^j K_2 \| u_0 \|_{L^\infty (\R^d)}} \qquad \text{and}
      \qquad 
       T^\ast_q \ge \sum_{j=1}^\infty \frac{\ln 2}{K_1 + 2^j K_2 \| u_0 \|_{L^q (\R^d)}},
$$
where $K_1$ and $K_2$ only depend on $L$, $d$, $N$,  $|D \eta| (\R^d)$, and $\| u_0 \|_{L^\infty}$ (respectively, on $L$, $d$, $N$,\break $\| \nabla \eta \|_{L^p (\R^d; \R^N \times \R^d)}$, and $\| u_0 \|_{L^q (\R^d)}$). 
In some cases, the nonlocal equation at the first line of~\eqref{e:nonlocal} possesses some further mechanism preventing blow up and as a consequence one has a global-in-time solution. This occurs for instance in the above-mentioned case of anisotropic kernels in one space dimension, where a maximum principle holds. Note however that, contrary to the case of smooth kernels where one always has global-in-time existence~\cite{CGLM,GianMag,KP}, in general in the rough-kernel case finite-time blow up does indeed occur, and we refer to~\cite[p.4065]{KP} and also to equation~\eqref{e:nonlocal2} below for an explicit example.  We now state our uniqueness result. 
\begin{theorem} \label{t:uni} 
Let $V: \R_+ \times \R^d \times \R^N \to \R^d$  be an $L$-Lipschitz continuous function 
and $T \in ]0, + \infty]$.  Then the following holds. 
\begin{itemize}
\item[(i)]  Assume that $\eta \in BV (\R^d;\R^N)$, $u_0 \in L^\infty (\R^d)$ and that either 
\begin{equation}
\label{e:etaR}
    \lim_{R \to + \infty} R^d \int_{\R^d \setminus B_R (0)} |\eta (x)| dx =0, \quad \text{where $B_R (0): = \{ x \in \R^d: \, |x| \leq  R \}$}, 
\end{equation}
or 
\begin{equation}
\label{e:elleuno}
         u_0 \in L^1 (\R^d)
\end{equation}
holds true. Then, there is at most one distributional solution  
$u \in L^\infty (]0, T[ ; L^\infty(\R^d))$ of the Cauchy problem~\eqref{e:nonlocal}. 
\item[(ii)] Fix $p>1$ and let $q = \sfrac{p}{(p-1)}$. Assume that $\eta \in W^{1, p} (\R^d;\R^N)$, $u_0 \in L^q (\R^d)$
 and either 
\begin{equation}
\label{e:etaRp}
    \lim_{R \to + \infty} R^d \left( \int_{\R^d \setminus B_R (0)} |\eta (x)|^p dx\right)^{1/p} =0 , \quad \text{where $B_R (0): = \{ x \in \R^d: \, |x| \leq  R \}$,}
\end{equation}
or~\eqref{e:elleuno} holds true. Then, there is at most one distributional solution 
$u \in L^\infty (]0, T[ ; L^q(\R^d))$ of the Cauchy problem~\eqref{e:nonlocal}. 
\end{itemize}
\end{theorem}
Interestingly, the proof of Theorem~\ref{t:uni} boils down to a Gr\"onwall-type argument applied to the functional $Q_R$ defined by equation~\eqref{e:qerre} in \S\ref{ss:puni} below. The use of similar functionals is by now fairly classical in the analysis of equations coming from fluid dynamics or kinetic theory, see for example~\cite{Loeper,MarchioroPulvirenti}, but to the best of our knowledge this is one of the first instances where it is used in the nonlocal conservation laws framework. The proof is quite flexible and can actually be extended up to the case $p=+ \infty$ and $u_0$ a finite measure, thus recovering results from~\cite{GianMag} (see also~\cite{DFFR} for a very recent contribution in the one-dimensional, anisotropic kernels case). 

To state our last result we consider the blow up example in~\cite[p.4065]{KP}, which corresponds to the choices $d=1$, $V(t, x, \xi) = \xi$, $\eta  = \mathbbm{1}_{]-1, 0[}$, and we choose $u_0 = 2 \mathbbm{1}_{]0, \sfrac{1}{2}[}$, so that the Cauchy problem~\eqref{e:nonlocal} boils down to
\begin{equation}
\label{e:nonlocal2}
     \left\{
     \begin{array}{ll}
    \partial_t u + \displaystyle{\partial_x \left[u  \left(\int_{x}^{x+1} u (t, y) dy \right)\right] =0,} \\
    \phantom{xx} \\
     u(0, x) = 2 \mathbbm{1}_{]0, \sfrac{1}{2}[} (x).\\
    \end{array}
     \right.
\end{equation}
The locally bounded solution of~\eqref{e:nonlocal2} is 
\be 
\label{e:formulaex}
      u(t, x) =  \frac{2}{1-2t} \mathbbm{1}_{\{(t, x) : \, t \leq x \leq \sfrac{1}{2}\}},
\eq
hence it is defined up to time $t=\sfrac{1}{2}$, and satisfies 
\be \label{e:blowup}
     u(t, \cdot) \weaks \delta_{x=\sfrac{1}{2}} \quad \text{weakly$^\ast$ in $\mathcal M(\R)$ as $t \uparrow \sfrac{1}{2}$},
\eq
where $\mathcal M(\R)$ denotes the space of finite Radon measures on $\R$ and $\delta_{x=\sfrac{1}{2}}$ the Dirac Delta concentrated at $x =\sfrac{1}{2}$.  Let us now consider the following question:  let $\{ \eta_n\}_{n \in \mathbb N}$ be a sequence of smooth kernels converging to $\mathbbm{1}_{]-1, 0[}$, and consider the corresponding sequence of Cauchy problems
\begin{equation}
\label{e:nonlocal3}
     \left\{
     \begin{array}{ll}
    \partial_t u_n + \partial_x [u_n (u_n \ast \eta_n) ] = 0, \\
     u_n(0, x) = 2 \mathbbm{1}_{]0, \sfrac{1}{2}[} (x). \\
    \end{array}
     \right.
\end{equation}
Global-in-time existence and uniqueness results for the above Cauchy problem are established in several papers, see for instance~\cite{CGLM,GianMag,KP}. Does the solutions sequence $\{ u_n \}_{n \in \mathbb N}$ admit a unique accumulation point? If the answer were positive, we could regard the limit point as a suitable definition of solution of~\eqref{e:nonlocal2} after the blow-up time. However, our last result shows that this is actually not the case, thus ruling out a fairly natural strategy for extending the solution after the blow up time.
\begin{theorem} \label{t:cex}
For every $\alpha \in [0, 1]$ there is a sequence $\{ \eta^\alpha_n \}_{n \in \mathbb N} \subseteq W^{1, \infty} (\R)$ such that
\be \label{e:etaconv}
\eta^\alpha_n \to \mathbbm{1}_{]-1, 0[} \quad \text{in $L^1(\R)$ as $n \to + \infty$}
\eq
and the corresponding sequence $\{ u^n_\alpha\}_{n \in \mathbb N}$ of solutions of~\eqref{e:nonlocal3} satisfies 
\begin{equation}\label{e:limitcex}
u^\alpha_n \weaks 
\left\{ \begin{array}{ll} u & \text{weakly$^\ast$ in $L^\infty(]0, T[\times \R)$ for every $T< \sfrac{1}{2}$,} \\
\nu_\alpha & \text{weakly$^\ast$ in $L^\infty (]\sfrac{1}{2}, +\infty[; \mathcal M (\R))$,}
\end{array}\right.
\end{equation}
as $n \to + \infty$, 
where $u$ is the locally bounded solution of~\eqref{e:nonlocal2} for $0 \leq t < \sfrac{1}{2}$, and
$\nu_\alpha$ is 
 a Dirac Delta concentrated on the line $x = \alpha t+ (1-\alpha)/2$, that is, 
\be \label{e:nualpha}
\nu_\alpha \in L^\infty (]\sfrac{1}{2}, + \infty[ ; \mathcal M(\R))\,,
\qquad
  \nu_\alpha(t):=   \delta_{\alpha t+ (1-\alpha)/2} \,.
\eq 
\end{theorem}
The construction of $\eta_n^\alpha$ is completely explicit (see formulas~\eqref{e:etaalpha} and~\eqref{e:errealpha} in \S\ref{ss:ualphan} below) and the very basic idea can be handwavingly described as follows. The convergence at the first line of~\eqref{e:limitcex} is fairly standard, the most intriguing part is establishing the convergence at the second line. We recall~\eqref{e:blowup} and 
conclude that, at the heuristic level, to extend the limit solution after the blow-up time, one should determine the velocity field in~\eqref{e:nonlocal3}, and hence give a meaning to the convolution $\delta_{x=\sfrac{1}{2}} \ast \mathbbm{1}_{]-1, 0[}$, which basically amounts to single out the value of the $BV$ function $\mathbbm{1}_{]-1, 0[}$ at $x=0$. For every $\alpha \in [0, 1]$, the basic idea underpinning the definition of $\eta^\alpha_n$ is heuristically speaking to select~$\alpha$ as the value attained by the limit $\eta$ at $x=0$, see Figure~\ref{f:2} in \S\ref{ss:ualphan} for a representation. This implies that, in the $n \to + \infty$ limit, the Dirac Delta concentrated at $(t, x)= (\sfrac{1}{2}, \sfrac{1}{2})$ propagates with velocity~$\alpha$, which in turn yields the convergence at the second line of~\eqref{e:limitcex}. 

\subsection*{Outline}The exposition is organized as follows. In \S\ref{s:pr} we recall some preliminary results, in \S\ref{s:ex} and~\S\ref{s:uni} we establish Theorems~\ref{t:ex} and~\ref{t:uni}, respectively. Finally, in \S\ref{s:cex} we provide the proof of Theorem~\ref{t:cex}.

\subsection*{Notation}
For the reader's convenience we conclude the introduction by collecting the main notation used in the present paper. 
We denote by $C(a_1, \dots, a_n)$ a constant only depending on the quantities~$a_1, \dots, a_n$. Its precise value can vary from occurrence to occurrence. 
We also employ the following notation:
\begin{itemize}
\item $\R_+: = [0, + \infty[$
\item $|x|$: the standard Euclidean norm of $x \in \R^d$
\item $B_R(0)$: the ball of radius $R$ and center at $0$ in $\R^d$
\item $|E|$: the Lebesgue measure of the measurable set $E \subseteq \R^d$
\item $\mathcal L^d$: the $d$-dimensional Lebesgue measure
\item $|D \eta| (\R^d)$: the total variation of the function $\eta \in BV (\R^d; \R^N)$, see~\cite{AFP} 
\item $\mathbbm{1}_E$: the characteristic function of the set $E$:
$$
   \mathbbm{1}_E (x) : = 
   \left\{
   \begin{array}{ll}
          1 & x \in E \\
          0 & x \notin E \\
   \end{array}
   \right.
$$
\item $\mathcal M(X)$: the space of finite Radon measures on $X$, see~\cite{AFP}
\item $C_c(X)$: the space of continuous and compactly supported functions on $X$
\item $C([0, T]; w-L^q(\R^d))$: the space of maps $t \mapsto u(t, \cdot)$ which are continuous from $[0, T]$ to~$L^q(\R^d)$ endowed with the weak topology 
\item $C([0, T]; w^\ast-L^\infty(\R^d))$: the space of maps $t \mapsto u(t, \cdot)$ which are continuous from $[0, T]$ to~$L^\infty(\R^d)$ endowed with the weak$^\ast$ topology 
\item $d$: the dimension of the space domain, $x \in \R^d$
\item $L$: the Lipschitz constant of the function $V$ in~\eqref{e:nonlocal}. 
\end{itemize}
\subsection*{Acknowledgments}
M.~Colombo is supported by the Swiss State Secretariat for Education, Research and lnnovation (SERI) under contract number MB22.00034 through the project TENSE.  G.~Crippa is supported by the Swiss National Science Foundation through the project 212573 FLUTURA (Fluids, Turbulence, Advection) and by the SPP 2410 ``Hyperbolic Balance Laws in Fluid Mechanics: Complexity, Scales, Randomness (CoScaRa)'' funded by the Deutsche Forschungsgemeinschaft (DFG, German Research Foundation) through the project 200021E\_217527 funded by the Swiss National Science Foundation. L.V. Spinolo is a  member of the GNAMPA group of INDAM and of the PRIN 
Project 20204NT8W4 (PI Stefano Bianchini), of the PRIN Project 2022YXWSLR (PI Paolo Antonelli), 
of the PRIN 2022 PNRR Project P2022XJ9SX (PI Roberta Bianchini) and of the CNR FOE 2022 Project STRIVE.

\section{Preliminary results} \label{s:pr}
First, we recall the definition of distributional solution of the Cauchy problem~\eqref{e:nonlocal}. 
\begin{definition}
Fix $T>0$, $u_0 \in L^1_{\mathrm{loc}} (\R^d)$ and assume that $u \in L^1_{\mathrm{loc}} (]0, T[ \times \R^d)$, $\eta \in L^1 _{\mathrm{loc}} (\R^d; \R^N)$ are such that $[u \ast \eta] \in L^\infty (]0, T[\times \R^d; \R^N)$. We term $u$ a distributional solution of the Cauchy problem~\eqref{e:nonlocal} if for every $\varphi \in C^\infty_c ([0, T[ \times \R^d)$ we have 
\be \label{e:dsol}
     \int_0^T \! \!  \int_{\R^d} u[ \partial_t \varphi + V(\cdot, \cdot, u\ast \eta) \cdot \nabla \varphi ] dx dt + \int_{\R^d} u_0 \varphi(0, \cdot) dx = 0. 
\eq
We also say that $u$ is a distributional solution of the equation at the first line of~\eqref{e:nonlocal} if \eqref{e:dsol} holds for every $\varphi \in C^\infty_c (]0, T[ \times \R^d)$ (so that the second integral in \eqref{e:dsol} does not appear).
\end{definition}
Next, we consider the linear continuity equation 
\be \label{e:conteq}
     \partial_t u + \mathrm{div} [bu ] =0,
\eq
where $b: \R_+ \times \R^d \to \R^d$ is a given vector field, and recall the following property (the proof is a slight extension 
of the proof of~\cite[Lemma 1.3.3]{Dafermos}). Fix $T>0$ and $q \in [ 1, + \infty]$
and assume that $b \in L^\infty_{\mathrm{loc}} (]0, T[ \times \R^d; \R^d)$ and $u \in 
L^\infty (]0, T[; L^q (\R^d))$ satisfy~\eqref{e:conteq} in the sense of distributions on $]0, T[ \times \R^d$. Then $u$ has a representative (namely, can be re-defined on a  $\mathcal{L}^{1+d}$-negligible set) so that the map $t \mapsto u(t, \cdot)$ is continuous from $[0, T]$ to $L^q(\R^d)$ endowed with the weak topology (or the weak$^\ast$ topology, if $q= + \infty$), that is,
$$
u \in C([0, T]; w-L^q(\R^d)) 
\qquad \text{ resp.} \qquad
u \in C([0, T]; w^\ast-L^\infty(\R^d))\,.
$$
\begin{remark} \label{r:dafermos}
Fix $T>0$, $p \in [1,+\infty]$, $q = \sfrac{p}{(p-1)}$,  and $V: \, ]0, T[  \times \R^d \times \R^N \to \R^d$ an $L$-Lipschitz continuous function. Assume that 
$$
    \eta \in L^p(\R^d; \R^N), \qquad u \in L^\infty (]0, T[ ; L^q( \R^d))
$$
satisfy the equation at the first line of~\eqref{e:nonlocal} in the sense of distributions. 
Then $V(\cdot, \cdot, u\ast \eta) \in L^\infty_{\loc} (]0, T[ \times \R^d; \R^d)$ and hence $u$ admits a representative (namely, can be re-defined on a  $\mathcal{L}^{1+d}$-negligible set) such that
$$
u \in C([0, T]; w-L^q(\R^d)) \quad \text{for $1 \leq q < \infty$,}
\quad \text{ resp.} \quad
u \in C([0, T]; w^\ast-L^\infty(\R^d)) \quad \text{for $q = \infty$.}
$$
In the following we will always tacitly identify $u$ with its continuous representative.
\end{remark}
In the following we need the following elementary result. 
\begin{lemma} \label{l:Lipcont}
Let $T>0$,  $p \in [1, + \infty]$, $q = \sfrac{p}{(p-1)}$, and  $V: \, ]0, T[  \times \R^d \times \R^N \to \R^d$ an $L$-Lipschitz continuous function.  
Consider either 
\be \label{e:wipeta}
   \eta \in W^{1, p} (\R^d; \R^N), \qquad
 u \in C([0, T]; w-L^q(\R^d))
\eq
 or 
\be \label{e:bveta} 
    \eta \in BV (\R^d; \R^N), \qquad u \in C([0, T]; w^\ast-L^\infty(\R^d)) 
\eq
and set 
\begin{equation}
\label{e:b}
        b(t, x) : = V(t, x, [u \ast \eta ](t, x)).
\end{equation}
Then the function $b: [0, T] \times \R^d \to \R^d$ is continuous, and it is Lipschitz continuous with respect to the $x$ variable, uniformly with respect to the $t$ variable. More precisely, 
\be \label{e:lipb2}
    \|\nabla_x b\|_{L^\infty (]0, T[ \times \R^d)} \leq L + L \| \nabla \eta \|_{L^p (\R^d; \R^{N\times d})}  \| u  \|_{L^\infty (]0, T[; L^q( \R^d))}, 
\eq
under~\eqref{e:wipeta}, and 
\be \label{e:lipb}
    \|\nabla_x b\|_{L^\infty (]0, T[ \times \R^d)} \leq L + L | D \eta |(\R^d)  \| u  \|_{L^\infty (]0, T[ \times \R^d)}, 
\end{equation}
under~\eqref{e:bveta}. 
\end{lemma}
Note  that \eqref{e:lipb} provides in particular an $L^\infty$ bound for $\div b$, which we will repeatetly use in the rest of this paper.
\begin{proof}
By the chain rule, for any $i,j=1,\ldots, d$ we have
\be \label{e:divb}
    \partial_{x_i}  b^j  = \partial_{x_i}  V^j +  \nabla_\xi V^j \cdot \partial_{x_i}[ \eta \ast u] ,  
\eq
where we denote by $\xi \in \R^N$ the third independent variable in the argument of $V$. Since 
$$
     \| \partial_{x_i} [ \eta \ast u] \|_{L^\infty} \leq\| \nabla \eta \|_{L^p}  \| u  \|_{L^\infty (]0, T[; L^q( \R^d))}, \quad 
\| \partial_{x_i} [ \eta \ast u] \|_{L^\infty } \leq |D \eta| (\R^d)  \| u  \|_{L^\infty}
$$
under~\eqref{e:wipeta} and~\eqref{e:bveta}, respectively, by using~\eqref{e:divb} we arrive at~\eqref{e:lipb2} and~\eqref{e:lipb}. Owing to the Lipschitz continuity of $V$, in order to conclude the proof of the lemma it suffices to establish the continuity in time of the convolution term $ [u \ast \eta ](t, x)$.
To this end, it suffices to observe that, for any given $(t,x) \in [0, T] \times \R^d$, we have
\begin{equation*}\begin{split}
    \lim_{s\to t} \big[ u \ast \eta(s,x)- u \ast \eta(t,x) \big] =    \lim_{s\to t}\int_{\R^d} [ u(s, y)  - u(t, y)]  \eta(x - y) dy = 0,
        \end{split}
\end{equation*}
by the weak continuity and the weak$^\ast$ continuity of $t \mapsto u(t, \cdot)$ tested againts $\eta(x - y)$,  under assumption~\eqref{e:wipeta} and~\eqref{e:bveta}, respectively.
\end{proof}
Under the same assumptions as in Lemma~\ref{l:Lipcont}, we now consider the characteristic line $X(t, x)$ defined by solving the Cauchy problem 
\be \label{e:21}
\left\{
\begin{array}{ll}
\displaystyle{\frac{d X}{dt}(t,x) = b(t,X) 
} 
\\ \phantom{c}\\
X(0, x) = x, \\
\end{array}
\right.
\eq
which is well-posed owing to the Cauchy-Lipschitz-Picard–Lindel\"of Theorem for ODEs combined with Lemma~\ref{l:Lipcont} above. This also yields the well-posedness of the Cauchy problem for the continuity equation~\eqref{e:conteq}. More precisely, for every $u_0 \in L^q (\R^d)$, there is a unique distributional solution $u \in L^\infty(]0, T[; L^q(\R^d))$ of the Cauchy problem~\eqref{e:nonlocal} and it verifies $u(t, \cdot) = X(t, \cdot)_\# u_0$, namely 
\begin{equation}
\label{e:pushf2}
\int_{\mathbb{R}^d} \varphi(x)u(t, x) \, dx = \int_{\mathbb{R}^d} \varphi(X(t,z))u_0(z) \, dz, 
\quad \text{for every $\varphi \in C^0_c (\R^d)$}
\end{equation}
and 
\begin{equation}
\label{e:repfor}
       u(t, X(t, x)) = u_0 (x) \exp \left( \int_0^t   \div b (s,  X(s, x)) ds  \right) \quad \text{for a.e. $(t, x) \in ]0, T[ \times \R^d$}.
\end{equation}
(see for instance~\cite[Section~2]{HW}). Finally, such solution is \emph{renormalized} and in particular it satisfies
$\partial_ t |u| + \div [b|u|] =0$ in the distributional sense. This implies that the unique solution of the continuity equation with initial datum $ |u_0|$ satisfies $X(t, \cdot)_\# |u_0| = \big| X(t, \cdot)_\# u_0 \big| = |u(t,\cdot)|$, and so we have 
\begin{equation}
\label{e:pushf}
\int_{\mathbb{R}^d} \varphi(x)|u(t, x)| \, dx = \int_{\mathbb{R}^d} \varphi(X(t,z))|u_0(z)| \, dz, 
\quad \text{for every $\varphi \in C^0_c (\R^d)$.}
\end{equation}
\section{Proof of Theorem~\ref{t:ex}} \label{s:ex}
The exposition is organized as follows: in \S\ref{ss:exi} we provide the proof of item~(i), and in \S\ref{ss:exii} we provide the proof of item~(ii). 
\subsection{Proof of Theorem~\ref{t:ex}, item (i)} \label{ss:exi}
In the case of a smooth convolution kernel $\eta$, existence and uniqueness results for the Cauchy problem~\eqref{e:nonlocal} are established for instance in~\cite{GianMag}. See also~\cite{CLM,KPS}. In the proof of the theorem we need the following preliminary lemma.
\begin{lemma} \label{l:exreg}
Let $p \in [1,+\infty]$ and $q = \sfrac{p}{(p-1)}$. Assume that $\eta \in C^\infty \cap W^{1, p} (\R^d;\R^N)$ and that $u_0 \in C^\infty_c (\R^d)$. Then there are  $K_1 = C(L)$, $K_2 = C(L, d, N) \| \nabla \eta \|_{L^p (\R^d;\R^{N \times d})}$ such that the solution of the Cauchy problem~\eqref{e:nonlocal}  satisfies
$$
     \| u \|_{L^\infty(]0,  T_q[ ; L^q(\R^d))} \leq 2 \| u_0 \|_{L^q(\R^d)} ,
$$
provided 
\be \label{e:T}
     T_q = \frac{\ln 2}{K_1 + 2 K_2 \| u_0 \|_{L^q (\R^d)}}.
\eq
\end{lemma}
\begin{proof}
By~\cite[eq. (2.2)]{GianMag}, the solution of~\eqref{e:nonlocal} satisfies~\eqref{e:repfor}, provided $b$ is as in~\eqref{e:b}. This implies in particular that $u \in C^0 (\R_+ \times \R^d)$, that $u(t, \cdot)$ is compactly supported for every $t\ge 0$ and that the map {$t \mapsto 
\| u(t, \cdot)\|_{L^q (\R^d)}$} is continuous. We now set 
   \be \label{e:tq}
    t_q : = \sup \{ t \in [0,  T_q]: \| u (s, \cdot) \|_{L^q (\R^d)} \leq 2 \| u_0 \|_{L^q (\R^d)}  \; \text{for every $s \in [0, t]$} \},
\eq
where $T_q$ is defined as in \eqref{e:T} and the constants $K_1$ and $K_2$ are to be chosen later.
We now assume by contradiction that $t_q< T_q$ (so that, by continuity, $ \| u (t_q, \cdot) \|_{L^q (\R^d)}= 2 \| u_0 \|_{L^q (\R^d)}$) and reach a contradiction, thereby showing that  $t_q=T_q$ and hence establishing the statement of the lemma.


We fix $\beta \in C^1(\R)$ and by multiplying the equation at the first line of~\eqref{e:nonlocal} times $\beta'(u)$ we obtain (renormalization formula)
$$
    \partial_t [\beta(u) ]+ \div  [b \beta (u) ] + \div  b \, [\beta'(u) u - \beta (u)] =0, \qquad b(t, x) : = V(t, x, [u \ast \eta](t, x) ).
$$
By applying the above equation with $\beta(u) = |u|^q$, $1 \leq q < \infty$, and space-time integrating we arrive at 
$$
     \int_{\R^d} |u(t_q, x)|^q dx ds \leq  \| u_0 \|^q_{L^q (\R^d)} + (q-1)  \| \div b \|_{L^\infty (]0, t_q[ \times \R^d)} 
\int_0^{t_q} \int_{\R^d} |u(s, x)|^q dx ds, 
$$
which owing to Gr\"onwall's Lemma yields 
\be \label{e:gron}
\begin{split}
     \| u(t_q, \cdot) \|_{L^q (\R^d)} & \leq \| u_0 \|_{L^q (\R^d)}  \exp \left( \frac{q-1}{q} \, t_q \| \div b \|_{L^\infty (]0, t_q[ \times \R^d)}  \right) \\
     & \leq \| u_0 \|_{L^q (\R^d)}  \exp \left( t_q \| \div b \|_{L^\infty (]0, t_q[ \times \R^d)}  \right) .
\end{split}
\eq
Note that by directly estimating~\eqref{e:repfor} we conclude that the bound~\eqref{e:gron} holds also in the case $q=\infty$. 
By~\eqref{e:lipb} we conclude that 
\begin{equation}\label{e:divb2} \begin{split} 
     \| \div b(s, \cdot) \|_{L^\infty (\R^d)} & \stackrel{\eqref{e:divb}}{\leq} \underbrace{ C(L)}_{:=K_1} + \underbrace{C(L, d, N) \| \nabla \eta \|_{L^p (\R^d;\R^{N \times d})}}_{: = K_2} \| u(s, \cdot) \|_{L^q (\R^d)}
    \\ &  \stackrel{\eqref{e:tq}}{\leq}K_1 + 2 K_2 \| u_0 \|_{L^q (\R^d)}, \qquad \text{for every $s \leq t_q$.}
\end{split}
\end{equation}
 By plugging the above formula into~\eqref{e:gron} we get 
$$
    \| u(t_q, \cdot) \|_{L^q (\R^d)} \leq \| u_0 \|_{L^q (\R^d)} \exp (t_q [K_1 + 2 K_2 \| u_0 \|_{L^q (\R^d)}] )
    \stackrel{t_q < T_q, \eqref{e:T}}{<} 2 \| u_0 \|_{L^q (\R^d)},
$$
which contradicts the equality $\| u (t_q, \cdot) \|_{L^q (\R^d)} = 2 \| u_0 \|_{L^q (\R^d)}$. 
\end{proof}
We are now ready to provide the proof of item (i) of Theorem~\ref{t:ex}. 
\begin{proof}[Proof of Theorem~\ref{t:ex}, item (i)]
We apply Lemma~\ref{l:exreg} with $q=+ \infty$ and fix $T < T_\infty$. Next, we proceed according to the following steps. \\
{\sc Step 1:} we consider two sequences $\{ u_{0n} \}_{n \in \mathbb N}$, $\{ \eta_n \}_{n \in \mathbb N}$ such that  
\be \label{e:u0n}
     u_{0n} \weaks u_0 \quad \text{weakly$^\ast$ in $L^\infty (\R^d)$}, \quad \| u_{0n} \|_{L^\infty (\R^d)} \leq \| u_{0} \|_{L^\infty (\R^d)}, \quad  u_{0n} \in C^\infty_c (\R^d)
\eq 
and 
\be \label{e:etan}
    \eta_n \to \eta  \quad \text{strongly in $L^1 (\R^d; \R^N)$}, \quad \| \nabla \eta_n \|_{L^1 (\R^d; \R^{N\times d})} \leq |D \eta| (\R^d), \quad  \eta_{n} \in C^\infty_c (\R^d).
\eq
We term $\{ u_n \}_{n \in \mathbb N}$ the solution of~\eqref{e:nonlocal} with convolution kernel $\eta_n$ and initial datum $u_{0n}$. By applying Lemma~\ref{l:exreg} we get 
\be \label{e:linfty}
    \| u_n \|_{L^\infty (]0, T[ \times \R^d)} \leq 2 \| u_{0n} \|_{L^\infty(\R^d)} \stackrel{\eqref{e:u0n}}{\leq} 2 \| u_{0} \|_{L^\infty(\R^d)}.
\eq
Owing to~\eqref{e:linfty}, up to subsequences 
\be \label{e:weaks}
    u_n \weaks u \, \quad \text{weakly$^\ast$ in $L^\infty (]0, T[ \times \R^d)$}. 
\eq
Assume for a moment that, as we will show in {\sc Step 2} below, we can  find a representative of the function  $u \in L^\infty (]0, T[ \times \R^d)$ (which a-priori is well-defined only up to negligible sets) such that 
\be \label{e:pointwise}
    u_n (t, \cdot) \weaks u(t, \cdot) \quad \text{weakly$^\ast$ in $L^\infty ( \R^d)$ for every $t \in ]0, T[$}. 
\eq
We deduce from \eqref{e:pointwise} that
\be \label{e:strong}
      V(\cdot, \cdot, u_n \ast \eta_n) \to  V(\cdot, \cdot, u \ast \eta) \; \quad \text{strongly in $L^1_{\mathrm{loc}}(]0, T[ \times \R^d; \R^d)$}.
\eq
Indeed, for every $\Omega \subseteq \R^d$ measurable and bounded we have 
\begin{equation} \label{e:Vl1}
\begin{split}
       \| V(\cdot, \cdot, u_n \ast \eta_n) & -  V(\cdot, \cdot, u \ast \eta) \|_{L^1 (]0, T[ \times \Omega; \R^d)} 
        \leq L \|u_n \ast \eta_n - u \ast \eta \|_{L^1(]0, T[ \times \Omega; \R^N)} \\
        & \leq  L{ \|u_n \ast \eta_n - u_n \ast \eta \|_{L^1(]0, T[ \times \Omega; \R^N)}} + 
        L {\|u_n \ast \eta - u \ast \eta \|_{L^1(]0, T[ \times \Omega; \R^N)}}. 
\end{split}
\end{equation}
To control the first term at the right-hand side of the above expression we point out that for a.e.~$(t, x) \in ]0, T[ \times \Omega$ we have 
$$
     |u_n \ast \eta_n - u_n \ast \eta| (t, x) = \left| \int_{\R^d} u_n (t, x-y) [\eta_n - \eta ] (y) \right| \stackrel{\eqref{e:linfty}}{\leq} 2 \| u_0 \|_{L^\infty (\R^d)} \| \eta_n - \eta \|_{L^1 (\R^d; \R^N)}
     \stackrel{\eqref{e:etan}}{\to} 0 .
$$
To control the second term at the right-hand side of~\eqref{e:Vl1}, we point out that by \eqref{e:pointwise} for every $(t, x) \in ]0, T[ \times \Omega$ we have 
$$
    | u_n \ast \eta - u \ast \eta| (t, x) = \left| \int_{\R^d} [u_n - u] (t, y)  \eta (x-y)  \right|    \stackrel{\eqref{e:pointwise}}{\to} 0 . 
$$
By Lebesgue's Dominated Convergence Theorem we get that both terms in the right-hand side of~\eqref{e:Vl1} vanish in $n \to + \infty$ limit. 
With \eqref{e:strong} in place, we can then pass to the limit in~\eqref{e:dsol} and get that $u \in L^\infty (]0, T[ \times \R^d)$ is a distributional solution of~\eqref{e:nonlocal}. 
 \\
{\sc Step 2:} we are left to establish~\eqref{e:pointwise}. The argument relies on fairly standard techniques (see for instance~\cite[Chapter 3]{Brezis} and~\cite[Lemma 1.3.3]{Dafermos}) and we provide it for the sake of completeness. We recall~\eqref{e:linfty} and term $B$ the closed ball of radius $C(T) \| u_0 \|_{L^\infty(\R^d)}$ and center at the origin in $L^\infty (\R^d)$, and we endow it with the weak-$^\ast$ topology. We fix a countable family 
$\{ \psi_k \}_{k \in \mathbb N} \subseteq C^\infty_c (\R^d)$ dense in the unit ball of $L^1 (\R^d)$; owing to the proof of ~\cite[Theorem 3.28]{Brezis} the function 
$$
    d(w, v) : = \sum_{k \in \mathbb N} \frac{1}{2^k } \left| \int_{\R^d} [w-v] (x) \psi_k (x) dx \right|
$$
is a distance metrizing the weak$^\ast$ topology on $B$.  We introduce the sequence of functions
\be \label{e:Fn}
    F_n: [0, T] \to B, \qquad F_n (t) = u_n (t, \cdot), 
\eq
where $u_n$ is the same as in~\eqref{e:weaks}. We now verify that the sequence $\{ F_n \}_{n \in \mathbb N}$ is equi-continuous with respect to the distance $d$. We fix $\delta>0$ and choose $k^\ast \in \mathbb N$ such that for every $n \in \mathbb{N}$ it holds
\begin{equation}\label{e:d21}
\begin{split}
      \sum_{k=k^\ast}^{+\infty} \frac{1}{2^k } \left| \int_{\R^d} [u_n (t_1, x) - u_n (t_2, x) ] \psi_k (x) dx \right| 
     \stackrel{\eqref{e:linfty}}{\leq} \sum_{k=k^\ast}^{+\infty} \frac{4}{2^k } \| u_0 \|_{L^\infty (\R^d) } \underbrace{\| \psi_k \|_{L^1 (\R^d) }}_{\leq 1} 
     \leq \frac{\delta}{2}. 
\end{split}
\end{equation}
By using the equation for $u_n$ and integrating by parts in space, recalling that $\psi_k \in C^\infty_c(\mathbb R^d)$, we get
\begin{equation} \label{e:d22}
\begin{split}
         \sum_{k=1}^{k^\ast-1} & \frac{1}{2^k } \left| \int_{\R^d} [u_n (t_1, x) - u_n (t_2, x) ]  \psi_k (x) dx \right| =
          \sum_{k=1}^{k^\ast-1}\frac{1}{2^k } \left| \int_{\R^d} \psi_k (x) \int_{t_2}^{t_1} \frac{\partial u_n}{\partial t} (s, x) ds dx \right| \\
           & 
           {=}   \sum_{k=1}^{k^\ast-1} \frac{1}{2^k } \left| \int_{t_2}^{t_1} \int_{\R^d} \nabla \psi_k (x) \cdot  V(s, x, u_n \ast \eta_n)u_n (s, x) dx ds \right| \\
          &
          \stackrel{\eqref{e:linfty}}{\leq}2 [M+ 2 L  \| u_0 \|_{L^\infty (\R^d)}]  \| u_0 \|_{L^\infty (\R^d)} |t_2 - t_1 | \sum_{k=1}^{k^\ast-1} \frac{1}{2^k } \| \nabla \psi_k \|_{L^1 (\R^d; \R^d)} \leq \frac{\delta}{2}
\end{split}
\end{equation}
provided $|t_2-t_1|$ is sufficiently small. In the previous expression we have used the shorthand notation
$$
    M: = \max \left\{ V(t, x, 0):  \; t \in [0, T], \; x \in \bigcup_{k=1}^{k^\ast - 1} \mathrm{supp} \, \psi_k    \right\}.
$$
We then have 
$$
    d(F_n(t_1), F_n (t_2) )=    \sum_{k=1}^{+\infty} \frac{1}{2^k } \left| \int_{\R^d} [u_n (t_1, x) - u_n (t_2, x) ] \psi_k (x) dx \right|  
    \stackrel{\eqref{e:d21},\eqref{e:d22}}{\leq}\! \! \! \delta \qquad \text{for every $n \in \mathbb N$} 
$$
provided $|t_1-t_2|$ is sufficiently small, that is, the sequence $\{ F_n \}_{n \in \mathbb N}$ is equi-continuous. We apply the Arzel\`a-Ascoli Theorem and conclude that, up to subsequences, $F_n \to F$ in $C^0([0, T]; B)$ for some limit function $F \in C^0([0, T]; B)$. We now define the function $w \in L^\infty (]0, T[ \times \R^d)$ by setting $w(t, \cdot) : = F(t)$ and fix a test function $\varphi \in L^1 (]0, T[ \times \R^d)$. We note that for a.e. $t \in ]0, T[$ we can test the weak convergence $ u_n (t, \cdot) \weaks w(t, \cdot)$ with $\varphi(t, \cdot) \in L^1 (\R^d)$ 
and by \eqref{e:linfty} and Lebesgue's Dominated Convergence Theorem for the time integral this implies 
$$
    \int_0^T \int_{\R^d} \varphi(t, x) u_n(t, x) dx dt \to \int_0^T \int_{\R^d} \varphi(t, x) w(t, x) dx dt, \quad \text{for every $\varphi \in L^1 (]0, T[ \times \R^d)$}.
$$
By~\eqref{e:weaks} and the uniqueness of the weak$^\ast$ limit this yields $w=u$, which eventually  implies~\eqref{e:pointwise}. 
\end{proof}
\subsection{Proof of Theorem~\ref{t:ex}, item (ii)} \label{ss:exii}
The proof of Theorem~\ref{t:ex} item (ii) follows the same lines as the proof of Theorem~\ref{t:ex} item (i), so we only provide a sketch. We construct sequences $\{ u_{0n} \}_{n \in \mathbb N}$ and $\{ \eta_n \}_{n \in \mathbb N}$ such that 
\begin{equation*}
\begin{split}
     & u_{0n} \to  u_0 \quad \text{strongly in $L^q (\R^d)$}, \quad \| u_{0n} \|_{L^q (\R^d)} \leq \| u_{0} \|_{L^q (\R^d)}, \quad  u_{0n} \in C^\infty_c (\R^d),\\
    &  \eta_n \to \eta  \quad \text{strongly in $L^p (\R^d; \R^N)$}, \quad \| \nabla \eta_n \|_{L^p (\R^d; \R^{N\times d})} 
      \leq \| \nabla \eta \|_{L^p (\R^d; \R^{N\times d})}, \quad  \eta_{n} \in C^\infty_c (\R^d).
\end{split}
\end{equation*}
We term $\{ u_n \}_{n \in \mathbb N}$ the corresponding solution of~\eqref{e:nonlocal} and recall that $T_q$ is given by~\eqref{e:T}. 
Owing to Lemma~\ref{l:exreg}, for every $T < T_q$ and up to subsequences we have 
$$
    u_n \rightharpoonup u \; \quad \text{weakly in $L^q (]0, T[ \times \R^d)$}
$$
for some limit function $u \in L^q (]0, T[ \times \R^d).$ To conclude the proof of item (ii) it suffices to show that 
$$ 
      V(\cdot, \cdot, u_n \ast \eta_n) \to  V(\cdot, \cdot, u \ast \eta) \; 
      \quad \text{strongly in $L^p_{\mathrm{loc}}(]0, T[ \times \R^d; \R^d)$}.
$$
Towards this end, we can follow the same argument as in {\sc Step 1} and {\sc Step 2} of the proof of item~(i). Note that when applying the Arzel\`a-Ascoli Theorem we consider the same sequence of functions as in~\eqref{e:Fn}, but this time $B$ denotes a suitable ball in $L^q(\R^d)$ endowed with the weak topology if $q>1$, and, if $q=1$, the set of finite Radon measures with total variation bounded by a suitable constant and endowed with the weak$^\ast$ topology. \qed

\section{Proof of Theorem~\ref{t:uni}}\label{s:uni}
The exposition is organized as follows: in \S\ref{ss:puni} we establish item (i) and in \S\ref{s:punii} we establish item~(ii). 
\subsection{Proof of Theorem~\ref{t:uni}, item (i)} \label{ss:puni}
We fix two solutions $u_1, u_2 \in L^\infty (]0, T[ \times \R^d)$ of the same Cauchy problem~\eqref{e:nonlocal}. Our aim is to show that $u_1\equiv u_2$. We  proceed according to the following steps. \\
{\sc Step 1:} we recall~\eqref{e:21} and term $X_1(t, x)$ and $X_2 (t, x)$ the solutions of the Cauchy problems
\begin{equation} \label{e:cauchy}
\left\{
\begin{array}{ll}
   \displaystyle{\frac{d X_1}{dt} = V(t, X_1, u_1 \ast \eta(X_1))} \\
   \phantom{x} \\
   X_1(0, x) =x \\
\end{array}
\right.
\quad \text{and} \quad 
\left\{
\begin{array}{ll}
   \displaystyle{\frac{d X_2}{dt} = V(t, X_2, u_2 \ast \eta(X_2))} \\
   \phantom{x} \\
   X_2(0, x) =x, \\
\end{array}
\right.
\end{equation}
respectively. As pointed out before, the well-posedness of the above Cauchy problems follows from Lemma~\ref{l:Lipcont}, combined with the classical Cauchy-Lipschitz-Picard-Lindel\"of Theorem for ODEs. Also, since 
\begin{equation} \label{e:x1x2}
\begin{split}
    \frac{d}{dt} |X_1 -X_2| & \leq |V(t, X_1, u_1 \ast \eta(X_1)) - V (t, X_2, u_2 \ast \eta(X_2))| \\
    & \leq L |X_1 - X_2| + L|u_1 \ast \eta(X_1) - u_2\ast \eta(X_2) | \\
    & \leq L |X_1 - X_2| + L [ \| u_1 \|_{L^\infty} +  \| u_2 \|_{L^\infty} ] \|\eta \|_{L^1}\,,
\end{split}
\end{equation}
by the Gr\"onwall Lemma we conclude that $|X_1 - X_2 | \in L^\infty_\loc([0,T[;L^\infty (\R^d))$. \\
{\sc Step 2:}  we fix $R>0$ and we set 
\begin{equation}
\label{e:qerre}
       Q_R (t) : = \int_{\R^d} \chi_R (x) |u_0(x)| |X_1(t, x) - X_2 (t, x)| dx, 
\end{equation} 
where 
\begin{equation}
\label{e:chiR}
         \chi_R(x) : =
         \left\{
         \begin{array}{ll}
         1 & \text{if $u_0 \in L^1 (\R^d)$,} \\
         \exp( \sfrac{-|x|}{R}) & \text{if $u_0 \notin L^1 (\R^d)$}.  \\
         \end{array}
         \right.
\end{equation}
Note that if $u_0 \in L^1 (\R^d)$ we are abusing notation since actually $\chi_R$, and henceforth $Q_R$, does not depend on $R$. Also, owing to {\sc Step 1} and to the assumption $u_0 \in L^\infty (\R^d)$,  the quantity $Q_R$ is finite for every $R>0$. 
 \\
{\sc Step 3:} assume that we have established the equality $Q_{R^\ast}(t)=0$ for some $R^\ast>0$ and every $t \in ]0, T[.$ We now show that this implies $u_1 \equiv u_2$. Towards this end, we point out that  the equality $Q_{R^\ast}(t)=0$ implies that   $ X_1(t, x) - X_2 (t, x) =0 $ {for a.e. $x \in  \{z: u_0(z) \neq 0\}$}.  
This yields $u_1 \equiv u_2$, since for every $\varphi \in C^\infty_c (\R^d)$
\begin{equation*}
\begin{split}
     & \int_{\R^d} \varphi (x) [u_1(t, x) - u_2 (t, x)] dx  \stackrel{\eqref{e:pushf2}}{=}
     \int_{\R^d} [\varphi (X_1(t, z)) -\varphi(X_2 (t, z))] u_0(z) dz
   = 0. 
\end{split}
\end{equation*}
{\sc Step 4:} we compute
\begin{equation} \label{e:Q}
\begin{split}
 \dfrac{dQ_R}{dt}& \stackrel{\eqref{e:cauchy},\eqref{e:qerre}}{\leq} \int_{\mathbb{R}^d} \chi_R  |u_0| |
 b_1(t,X_1) - b_2(t,X_2)| \, dx  \\
 &\leq \int_{\mathbb{R}^d} \chi_R    |u_0| 
 | b_1(t,X_1) - b_1(t,X_2)|
 \, dx+ 
 \int_{\mathbb{R}^d} \chi_R  |u_0| | b_1(t,X_2) - b_2(t,X_2)|\, dx
 =: I+II 
\end{split}
\end{equation}
where we have used the notation $b_i(t,x) = V(t, x, [u_i \ast \eta ](t, x))$ for $i=1,2$, which is consistent with~\eqref{e:b}.
To control the first integral at the right-hand side of \eqref{e:Q} we use the Lipschitz bound on $b_1$ given by \eqref{e:lipb} and conclude that
\be \label{e:uno} \begin{split}
 I \leq [L + L |D \eta| (\R^d) \| u_1 \|_{L^{\infty}} ] \int_{\mathbb{R}^d} \chi_R (x) |u_0(x)| |X_1(t,x) - X_2(t,x)| \, dx=[L + L |D \eta| (\R^d) \| u_1 \|_{L^{\infty}} ]
  Q_R(t) .
\end{split}
\eq
Since $u_1, u_2 \in L^\infty$, by approximation the equality~\eqref{e:pushf2} holds for every test function $\varphi \in L^1 (\R^d)$. This yields a control on the integrand in $II$, namely
\begin{equation} \label{e:IIf}
\begin{split}
 | b_1(t,X_2) - b_2(t,X_2)|
   &\leq L | (u_1 - u_2)\ast \eta |(t, X_2) = L
 \left| \int_{\mathbb{R}^d} \eta(X_2(t, x) - y)[u_1(t,y) - u_2(t,y)] \, dy \right| \\ &
\stackrel{\eqref{e:pushf2}}{=} L
 \left| \int_{\mathbb{R}^d} \left[ \eta(X_2(t, x) - X_1(t, y)) - \eta(X_2(t, x) - X_2(t, y)) \right] u_0(y) \, dy \right|.
\end{split}
\end{equation}
{\sc Step 5:} we conclude the proof of Theorem~\ref{t:uni} in the case $u_0 \in L^1 (\R^d)$. We recall that in this case $\chi_R \equiv 1$, plug~\eqref{e:IIf} into the expression for $II$ and change the order of integration. We arrive at 
\begin{equation} \label{e:due}
\begin{split}
    II & \leq L \int_{\mathbb{R}^d} |u_0(y)|
  \int_{\mathbb{R}^d}  |\eta(X_2(t, x) - X_1(t, y)) - \eta(X_2(t, x) - X_2(t, y)) | |u_0(x)|  dx  dy  \\
     & \stackrel{\eqref{e:pushf}}{=}
 L \int_{\mathbb{R}^d} |u_0(y)|
  \int_{\mathbb{R}^d}  |\eta(x - X_1(t, y)) - \eta(x- X_2(t, y)) | |u_2(t, x)|  dx  dy  \\ 
     & \leq L |D \eta| (\R^d)  \| u_2 \|_{L^\infty} \int_{\mathbb{R}^d} |u_0(y)| |X_1 (t, y) - X_2 (t, y)| dy 
     \stackrel{\chi_R \equiv 1 }{=}   L |D \eta| (\R^d)  \| u_2 \|_{L^\infty}  Q_R (t)
  \end{split}
\end{equation}
and by plugging~\eqref{e:uno} and~\eqref{e:due} into~\eqref{e:Q}, using Gr\"onwall's Lemma and recalling that $Q_R(0) =0$ we conclude that $Q_R(t) =0$ for every $t \in ]0, T[$ and $R>0$. \\
{\sc Step 6:}  we conclude the proof of Theorem~\ref{t:uni} under assumption~\eqref{e:etaR}. In this case we 
have to introduce a slightly more technical argument owing to the presence of the cutoff function in the definition of 
$Q_R(t)$. Quite loosely speaking, the basic idea underpinning our argument is that we employ a computation like~\eqref{e:due} in some ``big" region $ \left\{(x, y):  |y | \leq |x| + c_L R\right\}$, and control what is left by using~\eqref{e:etaR}.  We now provide the technical details: we set for any $x \in \R^d$  
\be \label{e:Ex}
       E(x) : = \left\{ y \in \R^d:  |y | \leq |x| + c_L R\right\},
\eq
where $c_L$ denotes a constant (only depending on the Lipschitz constant $L$), whose exact value will be determined in {\sc Step 7}. Owing to the equality $ \chi_R(x) = \chi_R(y) \exp 
\big(  \frac{|y|- |x|}{R} \big)
$ and to~\eqref{e:Ex} we have 
\be \label{e:stimachi}
     \chi_R (x) 
     {\leq} e^{c_L} \chi_R (y) \; \quad \text{for every $y \in E(x)$ and $x\in\R^d$}.
\eq
Next, we recall that $II$ is defined as in~\eqref{e:Q} and conclude that 
\begin{equation} \label{e:2}
\begin{split}
     II& \stackrel{\eqref{e:IIf}}{\leq} L \int_{\mathbb{R}^d} \chi_R (x) |u_0(x)|
 \left| \int_{\mathbb{R}^d} \left[ \eta(X_2(t, x) - X_1(t, y)) - \eta(X_2(t, x) - X_2(t, y)) \right] u_0(y) \, dy \right|  \, dx \\
 & \stackrel{\eqref{e:Ex}}{\leq}
   \underbrace{ L \int_{\mathbb{R}^d} \chi_R (x) |u_0(x)|
 \left| \int_{E(x)} \left[ \eta(X_2(t, x) - X_1(t, y)) - \eta(X_2(t, x) - X_2(t, y)) \right] u_0(y) \, dy \right|  \, dx}_{:=II_A}   \\
 &\qquad + \underbrace{L \int_{\mathbb{R}^d} \chi_R (x) |u_0(x)|
      \int_{\R^d \setminus E(x)} |\eta(X_2(t, x) - X_1(t, y))| |u_0(y) | dy \, dx}_{: = II_{B1}} \\ &\qquad + 
       \underbrace{L \int_{\mathbb{R}^d} \chi_R (x) |u_0(x)|
      \int_{\R^d \setminus E(x)} |\eta(X_2(t, x) - X_2(t, y))| |u_0(y) | dy \, dx}_{:= II_{B2}}.
\end{split}
\end{equation}
We now control $II_A$ by using the Fubini Theorem as follows:
\begin{equation} \label{e:2A}
\begin{split}
    &II_A  \stackrel{\eqref{e:stimachi}}{\leq}
      \! L  e^{c_L} \int_{\mathbb{R}^d} \!  |u_0(x)|
  \int_{E(x)} \! \!\chi_R (y) \left| \eta(X_2(t, x) - X_1(t, y)) - \eta(X_2(t, x) - X_2(t, y)) \right| |u_0(y)| dy  dx\\
    & \stackrel{\text{Fubini}}{\leq}  L  e^{c_L}  \int_{\mathbb{R}^d} \int_{\R^d} |u_0(x)|
       \chi_R (y) \left| \eta(X_2(t, x) - X_1(t, y)) - \eta(X_2(t, x) - X_2(t, y)) \right| |u_0(y)| \, dx \, dy \\
        & \stackrel{\eqref{e:pushf}}{=}
         L  e^{c_L}  \int_{\mathbb{R}^d} |u_0(y)| \chi_R(y) \int_{\R^d} \left| \eta(x - X_1(t, y)) - \eta(x - X_2(t, y)) \right| |u_2 (t, x)| dx \ dy \\
         & \leq L e^{c_L}  |D \eta | (\R^d) \| u_2 \|_{L^\infty} 
         \int_{\mathbb{R}^d} |u_0(y)| \chi_R(y) |X_1(t, y)) - X_2(t, y)| dy =  L e^{c_L} |D \eta | (\R^d) \| u_2 \|_{L^\infty} Q_R(t). 
\end{split}
\end{equation}
We now control $  II_{B1}$: 
\begin{equation} \label{e:BII1}
\begin{split}
    II_{B1} & \stackrel{\xi = X_2(t, x) - X_1 (t, y)}{=} L \int_{\mathbb{R}^d} \chi_R (x) |u_0(x)|   \int_{D(t, x)} |\eta(\xi)| |u_1(t, X_2(t, x) - \xi) | d\xi \, dx \\
     & \leq L \| u_1 \|_{L^\infty} \int_{\mathbb{R}^d} \chi_R (x) |u_0(x)|   \int_{D(t, x)} |\eta(\xi)| d \xi dx, 
    \end{split}
\end{equation}
provided 
\be \label{e:D}
    D(t, x) : = \{ \xi \in \R^d: \xi = X_2(t, x) - X_1 (t, y), \; \text{with $y \in \R^d \! \setminus \!  E(x)$} \}     \,.
\eq
Assume for a moment we have established the inclusion 
\be \label{e:include}
       D(t, x) \subseteq \R^d \! \setminus \! B_R (0) \quad \text{for every $t \in [0, 1]$ and $R \ge C (\| u_1 \|_{L^\infty}, \| u_2 \|_{L^\infty}, L, \| \eta \|_{L^1})$.}
\eq
Owing to~\eqref{e:BII1} we then have for any such $R$
\begin{equation} \label{e:2B1}
\begin{split}
      II_{B1} & \leq L \| u_1 \|_{L^\infty} \left( \int_{\mathbb{R}^d} \chi_R (x) |u_0(x)| \right) \left(  \int_{\R^d \setminus B_R(0)} |\eta(\xi)| d \xi \right) \\
      &
      \stackrel{\chi_R (x) =  \exp(-|x|/R)}{\leq}  C(d) L R^d \| u_1 \|_{L^\infty} \| u_0 \|_{L^\infty}  \left(  \int_{\R^d \setminus B_R(0)} |\eta(\xi)| d \xi \right).
\end{split}
\end{equation}
By relying on an analogous argument we show that 
\begin{equation} \label{e:2B2}
\begin{split}
      II_{B2} & \leq  C(d) L R^d \| u_2 \|_{L^\infty} \| u_0 \|_{L^\infty}  \left(  \int_{\R^d \setminus B_R(0)} |\eta(\xi)| d \xi \right).
\end{split}
\end{equation}
We now combine~\eqref{e:Q},~\eqref{e:uno},~\eqref{e:2},~\eqref{e:2A},~\eqref{e:2B1} and~\eqref{e:2B2}, use the Gr\"onwall Lemma and recall that $Q_R(0)=0$. This yields
\be \label{e:quasi}
    Q_R(t)  \leq R^d \left(  \int_{\R^d \setminus B_R(0)} |\eta(\xi)| d \xi \right) \mathcal{K}_1 \big[ \exp (\mathcal{K}_2 t ) - 1\big], 
\eq
provided $\mathcal K_1$ and $\mathcal K_2$ are suitable constants (which could be explicitly computed) only depending on $d$, $L$, $|D \eta|(\R^d)$, $\| u_1 \|_{L^\infty} $, $\| u_2 \|_{L^\infty} $ and 
$\| u_0 \|_{L^\infty} $.   To conclude, we point out that the map $R \mapsto Q_R(t)$ is monotone non-decreasing, which yields for any $R^\ast>0$
\be \label{e:decade}
   Q_{R^\ast} (t) \leq \lim_{R \to + \infty} Q_R(t) \stackrel{\eqref{e:etaR},\eqref{e:quasi}}{=}0, \quad \text{for every $t \in [0, 1]$} \eq 
and hence by {\sc Step 3} to the identity $u_1 \equiv u_2$ on $t \in ]0, 1[$. By iterating the argument, we establish the identity $u_1 \equiv u_2$ on the time interval $t \in ]0, T[$. \\
 {\sc Step 7:} to conclude the proof of item (i) we are thus left to establish~\eqref{e:include}, where $D$ is the same as in~\eqref{e:D}. Towards this end, we recall~\eqref{e:cauchy} and argue as in~\eqref{e:x1x2} to get
$$
\left| \frac{d}{dt} | X_1(t,y)-X_2(t,y)| \right| 
\leq
L |X_1(t,y) - X_2(t,x)| + L [ \| u_1 \|_{L^\infty} +  \| u_2 \|_{L^\infty} ] \|\eta \|_{L^1}\,,
$$
which upon integration in time implies 
%
\begin{equation*}
\begin{split}
     |X_1 (t, y)  - X_2 (t, x)|  
     & \ge  |y-x| e^{-Lt} -   [ \| u_1 \|_{L^\infty} +  \| u_1 \|_{L^\infty} ] \|\eta \|_{L^1} [1-e^{-Lt}] \\
     &  \stackrel{t \leq 1}{\geq}   |y-x| e^{-L} -   [ \| u_1 \|_{L^\infty} +  \| u_1 \|_{L^\infty} ] \|\eta \|_{L^1} [1-e^{-L}] \\ & 
     \stackrel{y \in \R^d \setminus  E(x)} {\ge} c_L R e^{-L}    -   [ \| u_1 \|_{L^\infty} +  \| u_1 \|_{L^\infty} ] \|\eta \|_{L^1} [1-e^{-L}] 
       \end{split}
\end{equation*}
for $y \in \R^d \! \setminus \! E(x) $. Choosing $c_L$ and $R$ so that
$$
   c_L e^{-L} \ge \frac{3}{2}, \qquad R \ge 2  [ \| u_1 \|_{L^\infty} +  \| u_1 \|_{L^\infty} ] \|\eta \|_{L^1} [1-e^{-L}] 
$$
we conclude that~\eqref{e:include} holds. \qed

 \subsection{Proof of Theorem~\ref{t:uni}, item (ii)} \label{s:punii}
 The proof follows the same steps as the one of item (i), so rather than repeating it we only highlight the points where the argument is different. In~\eqref{e:x1x2} we control the term  $|[u_1 - u_2] \ast \eta| $  by $[\| u_1 \|_{L^q} + \| u_2 \|_{L^q}]\| \eta \|_{L^p}$ rather than by $[\| u_1 \|_{L^\infty} + \| u_2 \|_{L^\infty}]\| \eta \|_{L^1}$. In~\eqref{e:uno} we use {H\"older} inequality and the finite difference quotients characterization of Sobolev functions 
 and get  
 \begin{equation*} \begin{split}
  \Bigg| \int_{\mathbb{R}^d} &  [\eta(X_1 - y) - \eta(X_2 - y)]u_1(t, y) \, dy \Bigg|
  {\leq} 
  \|\eta(X_1 - \cdot) - \eta(X_2 - \cdot)\|_{L^p}
  \| u_1 \|_{L^q} 
{\leq}
  \| \nabla \eta \|_{L^p } |X_1 - X_2|  \| u_1 \|_{L^q}
 \end{split}
 \end{equation*}
 and similarly in~\eqref{e:due} and~\eqref{e:2A}.  To control term $II_{B1}$ we use, instead of~\eqref{e:BII1}, the estimate 
 \begin{equation} \label{e:BII1p}
\begin{split}
    II_{B1} & \stackrel{\xi = X_2(t, x) - X_1 (t, y)}{=} L \int_{\mathbb{R}^d} \chi_R (x) |u_0(x)|   \int_{D(t, x)} |\eta(\xi)| |u_1(t, X_2(t, x) - \xi) | d\xi \, dx \\
     & \leq L \| u_1 (t, \cdot)\|_{L^q} \int_{\mathbb{R}^d} \chi_R (x) |u_0(x)| \left(  \int_{D(t, x)} |\eta(\xi)|^p d \xi \right)^{1/p} \!  dx , 
    \end{split}
\end{equation}
 and similarly for $II_{B_2}$. Finally, in~\eqref{e:decade} we use~\eqref{e:etaRp} rather than~\eqref{e:etaR}. \qed

\section{Proof of Theorem~\ref{t:cex}} \label{s:cex}
The exposition is organized as follows: in \S\ref{ss:ualphan} we discuss some preliminary results and in \S\ref{ss:proof13} we complete the proof of Theorem~\ref{t:cex}. 
\subsection{Preliminary results}\label{ss:ualphan}
We recall the discussion in~\cite[p.4065]{KP}, apply Theorem~\ref{t:uni} and conclude that  for every $T < \sfrac{1}{2}$, the Cauchy problem~\eqref{e:nonlocal2} has a unique solution $u \in L^\infty(]0, T[ \times \R)$, which is given by formula~\eqref{e:formulaex}.  Next, we fix $\alpha \in [0, 1]$ and set 
\be \label{e:etaalpha}
      \eta^\alpha_n(x) : = \left\{
      \begin{array}{ll}
      0 & x \leq - r_\alpha - \sfrac{1}{n}  \\
      n x + n r_\alpha + 1 & - r^\alpha_n - \sfrac{1}{n} \leq x \leq - r^\alpha_n \\
      1 & - r^\alpha_n \leq x \leq \sfrac{-2}{n} \\
       n(\alpha -1 ) x + 1 + n (\alpha-1) \quad & \sfrac{-2}{n} \leq x \leq \sfrac{-1}{n} \\
       \alpha & \sfrac{-1}{n} \leq x \leq \sfrac{1}{n} \\
       - \alpha n x + 2 \alpha & \sfrac{1}{n} \leq x \leq \sfrac{2}{n} \\
       0 & x \ge \sfrac{2}{n}, \\
      \end{array}
      \right.
\eq
 where $r^\alpha_n>0$ is a suitable constant to be chosen in such a way that 
 \be \label{e:errealpha}
     \int_{\R} \eta^\alpha_n (x) dx =1 . 
 \eq 
See Figure~\ref{f:2} for a representation of $\eta_n^\alpha$. Note that by combining~\eqref{e:etaalpha} and~\eqref{e:errealpha} one gets~\eqref{e:etaconv}. We now consider the Cauchy problem~\eqref{e:nonlocal3}, with $\eta_n^\alpha$ in place of $\eta_n$, and term $u^n_\alpha$ its solution. More precisely, owing to the regularity of the kernel $\eta^n_\alpha$ it is well known that there is a unique $u^\alpha_n \in L^\infty_{\mathrm{loc}} (\R_+ \times \R)$ solving~\eqref{e:nonlocal3}, see for instance~\cite{CGLM,GianMag,KP}.  
To establish~\eqref{e:limitcex} we need the following result. 
 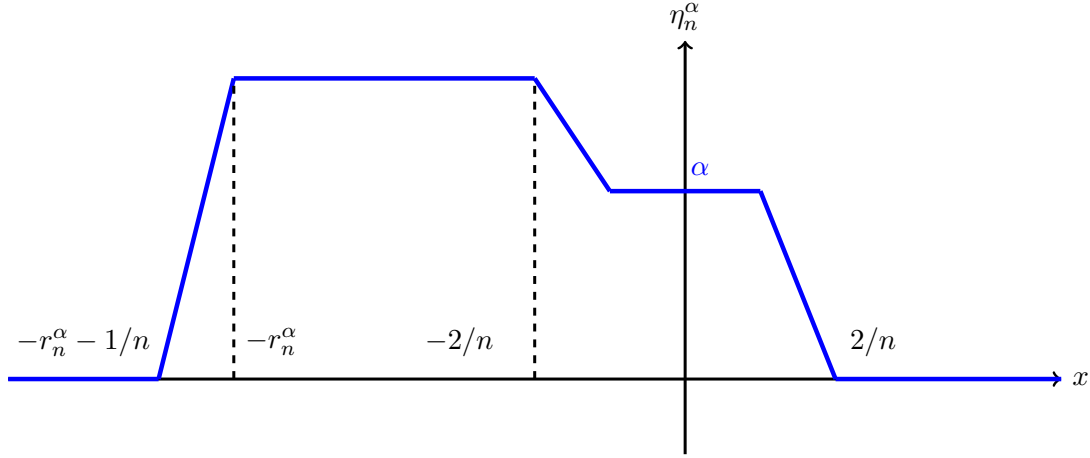
\begin{figure}
\begin{center}
\caption{The function $\eta^\alpha_n$ defined as in~\eqref{e:etaalpha}} 
\label{f:2}
\begin{tikzpicture}
\draw[line width=0.4mm, ->]   (-7, 0) -- (7, 0) node[anchor=west] {$x$}; 
\draw[line width=0.4mm, ->]   (2, -1) -- (2, 4.5) node[anchor=south] {$\eta^\alpha_n$}; 
\draw[line width=0.6mm, blue]   (-7, 0) -- (-5, 0);
\draw[line width=0.6mm, blue]   (-5, 0)--(-4, 4);
\draw[line width=0.6mm, blue]   (-4, 4)--(0,4);
\draw[line width=0.6mm, blue]   (0, 4)--(1,2.5);
\draw[line width=0.6mm, blue]   (1,2.5)--(3, 2.5);
\draw[line width=0.6mm, blue]   (3, 2.5)--(4,0);
\draw[line width=0.6mm, blue]   (4,0)--(7,0);
\draw[dashed, line width=0.4mm]   (-4,0)--(-4, 4);
\draw[dashed, line width=0.4mm]   (0,0)--(0, 4);
\draw  (-3.5,0.5) node {$- r^\alpha_n$};
\draw  (-6,0.5) node {$- r^\alpha_n-1/n$};
\draw  (-1,0.5) node {$- 2/n$};
\draw  (4.5,0.5) node {$ 2 /n$};
\draw  (2.2, 2.8) node {\color{blue}$\alpha$};
\end{tikzpicture}
\end{center}
\end{figure}
\begin{lemma} \label{l:linftyalpha}
Let $u_n^\alpha$ be the solution of the Cauchy problem~\eqref{e:nonlocal3}, with $\eta_n^\alpha$ defined in~\eqref{e:etaalpha},~\eqref{e:errealpha} in place of $\eta_n$. Then 
\be \label{e:linftyalpha}
     0 \leq u_n^\alpha (t, x) \leq \frac{2}{1-2t} \quad \text{for every $\alpha \in [0, 1]$, $n \in \mathbb N$ and a.e. $(t, x) \in [0, \sfrac{1}{2}[ \times \R$.}
\eq
\end{lemma}
\begin{proof}
We proceed according to the following steps. \\
{\sc Step 1:} we point out that~\eqref{e:b},~\eqref{e:21} boil down to
\be \label{e:charalphan}
\left\{
\begin{array}{ll}
\displaystyle{\frac{d X^\alpha_n}{dt}=  [\eta^\alpha_n \ast u^\alpha_n] (t, X^\alpha_n)} \\
\phantom{c}\\
X^\alpha_n (0, x) = x
\end{array}
\right.
\eq
and
\be \label{e:415} \begin{split}
   u^\alpha_n (t, X_n^\alpha(t, y)) \stackrel{\eqref{e:charalphan}}{=}     2 \mathbbm{1}_{[0, \sfrac{1}{2}]}(y)
    \exp \left( - \int_0^t  [(\eta^\alpha_n)' \ast u^\alpha_n]  (s, X_n^\alpha(s, y)) \right),
\end{split}
\eq
respectively. This immediately yields the first inequality (non-negativity) in~\eqref{e:linftyalpha}, and also the following implication:
\be \label{e:implica}
      \text{for every $t \in \R_+$, $\alpha \in [0, 1]$ and $n \in \mathbb N$}, \;u^\alpha_n (t, x) >0 \implies X^\alpha_n (t, 0) \leq x \leq X^\alpha_n (t, \sfrac{1}{2}).
\eq
{\sc Step 2:} we set for any $n \in \mathbb{N}$
\be \label{e:gi}
\begin{split}
      g(t) : & = \sup \{ u^\alpha_n (t, x), \; x \in \R \}   \stackrel{\eqref{e:implica}}{=}  \sup \{ u^\alpha_n (t, x), \;  X^\alpha_n (t, 0) \leq x \leq  X^\alpha_n (t, \sfrac{1}{2}) \} \\&
    =   \sup \{ u^\alpha_n (t, X^\alpha_n(t, y)), \; 0 \leq y \leq  \sfrac{1}{2} \}= \max 
     \{ u^\alpha_n (t, X^\alpha_n(t, y)), \; 0 \leq y \leq  \sfrac{1}{2} \} 
\end{split}
\eq 
and point out that the map $g$ is locally Lipschitz continuous on $\R$, since it is the maximum of uniformly Lipschitz in time functions by~\eqref{e:415}. Our goal is now to show that 
\be \label{e:gg2}
   g'(t) \leq g^2 (t) \quad \text{for a.e. $t \in ]0, \sfrac{1}{2}[$}. 
\eq
With~\eqref{e:gg2} in place, the inequality \eqref{e:linftyalpha} follows by the classical Comparison Theorem for ODEs and the equality $g(0)=2$. To establish~\eqref{e:gg2}, 
we fix a point $t_\ast \in ]0, \sfrac{1}{2}[$ at which $g$ is differentiable and by using~\eqref{e:gi} we conclude that 
\be \label{e:gast}
g(t_\ast) = u^\alpha_n (t_\ast, x_\ast) \quad  \text{for some $x_\ast \in [X^\alpha_n(t_\ast, 0), X^\alpha_n (t_\ast, \sfrac{1}{2})]$.}
\eq
We now let $a_\ast \in [0,\sfrac{1}{2}]$ such that $X^\alpha_n (t_\ast,a_\ast) = x_\ast$ and set
\be \label{e:effe}
   f(t) : = u^\alpha_n (t, X^\alpha_n (t,a_\ast)),
\eq 
and observe that $f (t) \leq g(t)$ for every $t \in \R+$. Note that the derivative of $f$ is the material derivative 
\be \label{e:derivoeffe}
     f'(t)= - u^\alpha_n \partial_x [u^\alpha_n \ast \eta^\alpha_n] =  - u^\alpha_n  [u^\alpha_n \ast (\eta^\alpha_n)']
\eq
and hence $f$ is a differentiable  function. We rely on the following elementary property:
\emph{Assume that $I \subseteq \R$ is  an interval, and $f, g: I \to \R$ are two  continuous functions with $f \ge g$. If both $f$ and $g$ are differentiable at $\tau \in I$ and $f(\tau) = g(\tau)$, then $f'(\tau) = g'(\tau)$. } (Indeed, set $h := f-g \ge 0$ and observe that if $h(\tau) =0$ then $\tau$ is a point of minimum $h$ and therefore by the differentiability of $h$ at $\tau$ we get $h'(\tau)=0$). It follows  that in order to establish the inequality $g'(t_\ast) \leq g^2 (t_\ast)$ it suffices to show that 
\be \label{e:effegquadro} 
   f'(t_\ast) \leq g^2(t_\ast). 
\eq
Towards this end, we  point out that by combining~\eqref{e:etaalpha} and~\eqref{e:errealpha} an elementary computation yields
\be \label{e:bderre}
      r^\alpha_n \ge 3/4 \quad \text{for every $n \ge 18$ and $\alpha \in [0, 1]$}.
\eq
We now show that 
\be \label{e:419}
      X^\alpha_n(t, \sfrac{1}{2}) - X^\alpha_n(t, 0) \leq \sfrac{3}{4}, \quad \text{for every $t \in \R_+$, $n \ge 18$ and $\alpha \in [0, 1]$}.
\eq
Note that by combining~\eqref{e:bderre} and \eqref{e:419} with~\eqref{e:etaalpha} we get
\be \label{e:420}
    ( \eta^\alpha_n)' ( x -y) \leq 0, \; \text{for every $t\geq 0$,  $x, y  \in [X^\alpha_n(t, 0), X^\alpha_n (t, \sfrac{1}{2})]$}.
\eq
To establish~\eqref{e:419} we write a differential inequality for the quantity $X^\alpha_n(t, \sfrac{1}{2}) - X^\alpha_n(t, 0)$. Indeed, we point out that $X^\alpha_n(0, \sfrac{1}{2}) - X^\alpha_n(0, 0)=\sfrac{1}{2}$ and that as long as this quantity is smaller than $3/4$ we have~\eqref{e:420} and hence  
\begin{equation}\label{e:427}
\begin{split}
    \frac{d [X^\alpha_n(t, \sfrac{1}{2}) - X^\alpha_n(t, 0)]}{dt}& \stackrel{\eqref{e:charalphan}}{=} \int_{\R}
     \big[ \eta_n^\alpha (X^\alpha_n(t, \sfrac{1}{2}) -y) - \eta_n^\alpha (X^\alpha_n(t, 0) -y) \big] u^\alpha_n (t, y) dy 
     \\ & \stackrel{\eqref{e:implica}}{=}
     \int_{X^\alpha_n(t, 0)}^{X^\alpha_n(t, \sfrac{1}{2})}
     \big[ \eta_n^\alpha (X^\alpha_n(t, \sfrac{1}{2}) -y) - \eta_n^\alpha (X^\alpha_n(t, 0) -y) \big] u^\alpha_n (t, y) dy \\
    & =
     \int_{X^\alpha_n(t, 0)}^{X^\alpha_n(t, \sfrac{1}{2})} u^\alpha_n(t, y)  \int_{X^\alpha_n(t, 0)}^{X^\alpha_n(t, \sfrac{1}{2})} (\eta_n^\alpha)'(z-y) dz dy  
     \stackrel{\eqref{e:420}}{\leq} 0.
\end{split}
\end{equation}
This establishes~\eqref{e:419}. We then point out that 
\begin{equation*}
\begin{split}
     f'(t_\ast) & \stackrel{\eqref{e:derivoeffe}}{=} - u^\alpha_n (t_\ast, x_\ast) [(\eta^\alpha_n)' \ast u^\alpha_n](t_\ast, x_\ast) \stackrel{\eqref{e:gast}}{=} - g(t_\ast) \int_{\R} (\eta^\alpha_n)' (x_\ast - y) u^\alpha_n (t_\ast, y) dy \\ 
& \stackrel{\eqref{e:implica}}{=} g(t_\ast) 
\int_{X^\alpha_n(t_\ast, 0)}^{X^\alpha_n(t_\ast, \sfrac{1}{2})}\underbrace{ (- \eta^\alpha_n)' (x_\ast - y)}_{\ge 0 \, \text{by \eqref{e:420}}} u^\alpha_n (t_\ast, y) dy 
   \stackrel{\eqref{e:gi}}{\leq} g^2(t_\ast) 
\int_{X^\alpha_n(t_\ast, 0)}^{X^\alpha_n(t_\ast, \sfrac{1}{2})} (- \eta^\alpha_n)' (x_\ast - y) dy \\
 & = g^2(t_\ast) \big[ \eta^\alpha_n (x_\ast  -   X^\alpha_n(t_\ast, \sfrac{1}{2})) -  
        \eta^\alpha_n (x_\ast -   X^\alpha_n(t_\ast, 0))  \big] \stackrel{\eta^\alpha_n \leq 1}{\leq} 
       g^2(t_\ast),
\end{split}
\end{equation*}
that is~\eqref{e:effegquadro}. 
\end{proof}
\subsection{Proof of Theorem~\ref{t:cex}} \label{ss:proof13}
We fix $\alpha \in [0, 1]$. By using Lemma~\ref{l:linftyalpha} and by arguing as in the proof of Theorem~\ref{t:ex}, item (i), we obtain the first limit in~\eqref{e:limitcex} (the one on the time interval $]0, \sfrac{1}{2}[$).  Establishing the second limit in~\eqref{e:limitcex} amounts to show that  
\be \label{e:429}
     \lim_{n \to + \infty}
     \int_{\sfrac{1}{2}}^{T} \int_{\R} \varphi(t, x) 
u^\alpha_n (t, x) dx dt =  \int_{\sfrac{1}{2}}^{T} \varphi \! \left(t, \alpha t+ \frac{1-\alpha}{2} \right) dt \quad 
    \text{for every $\varphi \in L^1 (]\sfrac{1}{2}, + \infty[; C^0 (\R)$).}
\eq
Since the left-hand side can be rewritten as
\begin{equation*}
\begin{split}
         \int_{\sfrac{1}{2}}^{T}  \int_{\R} \varphi(t, x) 
u^\alpha_n (t, x) dx dt  & \stackrel{\eqref{e:pushf2}}{=}  \int_{\sfrac{1}{2}}^{T}  \int_{\R} \varphi(t, X^\alpha_n(t, y)) 
u_0(y) dy dt =  2 \int_{\sfrac{1}{2}}^{T}  \int_0^{\sfrac{1}{2}} \varphi(t, X^\alpha_n(t, y)) dy dt,
\end{split}
\end{equation*}
by using the elementary bound 
$$
     \left| \int_0^{\sfrac{1}{2}} \varphi(t, X^\alpha_n(t, y)) dy \right| \leq \frac{1}{2}  \| \varphi (t, \cdot) \|_{C^0}
$$
we conclude by Lebesgue's Dominated Convergence Theorem that in order to establish~\eqref{e:429} it suffices to show that
\be \label{e:inmezzo}
     \lim_{n \to \infty} X^\alpha_n (t, y) = \alpha t + \frac{1-\alpha}{2}, \quad \text{for every $y \in [0, \sfrac{1}{2}]$ and $t \ge \sfrac{1}{2}$}. 
\eq
Towards this end, we first establish compactness of the sequence $\{  X^\alpha_n (\cdot, 0) \}_{n \in \mathbb N}$ in the $C^0([0, T])$ topology.  Since the spatial $L^1$ norm of $u^\alpha_n$ is constant in time, we point out that
\be \label{e:equic}
    0 \stackrel{\eqref{e:linftyalpha}}{\leq} [u^\alpha_n \ast \eta^\alpha_n](t, x) \leq \| u^\alpha_n \|_{L^1}
     \| \eta^\alpha_n \|_{L^\infty} \stackrel{\eqref{e:nonlocal3},\eqref{e:etaalpha}}{=} 1, \quad \text{for every $(t, x) \in \R_+$, $n \in \mathbb N$}.
\eq
Next, we recall~\eqref{e:charalphan} and conclude that $0 \leq X^\alpha_n (t, 0) \leq T$ for every $n \in \mathbb N$, $t \in [0, T]$, and that the sequence $\{  X^\alpha_n (\cdot, 0) \}_{n \in \mathbb N}$ is equi-Lipschitz. We apply the Arzel\`a-Ascoli Theorem and conclude that, up to subsequences (which to simplify the notation we do not relabel), $X^\alpha_n (\cdot, 0)$ converges to some limit function $X^\alpha_\ell$ in the $C^0([0, T])$ topology. 
Similarly, up to subsequences, $\{ X^\alpha_n (\cdot,  \sfrac{1}{2})\}_{n \in \mathbb N}$ converges to some limit function $X^\alpha_r$ in the $C^0([0, T])$ topology.  Since $X^\alpha_n (t, 0) \leq X^\alpha_n (t, y) \leq X^\alpha_n (t, \sfrac{1}{2})$ for every $t \in \R_+$ and $y \in [0, \sfrac{1}{2}]$ and for every $n \in \mathbb{N}$, in order to establish~\eqref{e:inmezzo} it suffices to show that 
\be \label{e:inmezzo3}
     X^\alpha_\ell (t) \ge \alpha t + \frac{1-\alpha}{2}, \quad 
      X^\alpha_r (t) \leq \alpha t + \frac{1-\alpha}{2}, 
     \quad \text{for every $t \in [\sfrac{1}{2}, T]$}. 
\eq
We split the ODE arguments leading to \eqref{e:inmezzo3} in several steps.
\\
{\sc Step 1:} we plug~\eqref{e:420} into~\eqref{e:415} and conclude that 
\be \label{e:cresce}
     u^\alpha_n (t, x) \ge 2, \quad  \quad \text{for every $t \in \R_+$, $n \ge 18$ and $x \in [X^\alpha_n (t, 0), X^\alpha_n (t,  \sfrac{1}{2})]$.}
\eq
{\sc Step 2:} we show that $X^\alpha_\ell(\sfrac{1}{2}) = \sfrac{1}{2} = X_r (\sfrac{1}{2})$. Owing to~\eqref{e:equic} we have 
\be \label{e:dovestanno}
     X^\alpha_n (t, 0) \leq t, \quad X^\alpha_n (t, \sfrac{1}{2})\ge \sfrac{1}{2} \quad \text{for every $n \in \mathbb N$ and $t \in \R_+$} 
\eq
and hence $X^\alpha_\ell(\sfrac{1}{2}) \leq \sfrac{1}{2}$, $X^\alpha_r(\sfrac{1}{2})\ge \sfrac{1}{2}$. Assume by contradiction that 
$X^\alpha_\ell(\sfrac{1}{2}) < \sfrac{1}{2}$; this implies that there is $\ee>0$ such that
\begin{equation} \label{e:asinistra}
    X^\alpha_n (\sfrac{1}{2}, 0) \leq  \sfrac{1}{2}  -\ee < \sfrac{1}{2}  \leq X^\alpha_n (\sfrac{1}{2}, \sfrac{1}{2})
\end{equation}
for $n$ sufficiently large. We set 
$$
    E_\ee := \{ (t, x):  \sfrac{1}{2} - \sfrac{\ee}{4} < t < \sfrac{1}{2} - \sfrac{\ee}{8} \,, \;\;  \sfrac{1}{2} - \ee < x< \sfrac{1}{2}{ - \sfrac{\ee}{2} } \}
$$
and we point out that, owing to~\eqref{e:equic} and~\eqref{e:asinistra}, we have 
$$
     X^\alpha_n (t, 0) \leq x \leq  X^\alpha_n (t, \sfrac{1}{2}) \quad \text{for every $(t, x)\in E_\ee$}
$$
and every $n$ sufficiently large. Owing to~\eqref{e:cresce}, this implies 
\begin{equation} \label{e:iint}
    \iint_{E_\ee} u^\alpha_n (t, y) dtdy \ge \ee^2/8
\end{equation}
for every $n$ sufficiently large. On the other hand, owing to~\eqref{e:formulaex}, we have 
\begin{equation} \label{e:iint2} 
    \iint_{E_\ee} u (t, y) dt dy=0.
\end{equation}
By combining~\eqref{e:iint},\eqref{e:iint2} and the limit at the first line of~\eqref{e:limitcex} we get a contradiction, which allows us to conclude that $X^\alpha_\ell(\sfrac{1}{2}) = \sfrac{1}{2}$. An analogous argument yields the equality 
$X^\alpha_r(\sfrac{1}{2}) = \sfrac{1}{2}$. \\
{\sc Step 3:} we conclude the proof of~\eqref{e:inmezzo3}. We recall that, owing to~\eqref{e:charalphan}, 
\begin{equation} \label{e:2D} 
\begin{split}
    X^\alpha_n (t, 0) & =  X^\alpha_n (\sfrac{1}{2}, 0)
    + \int_{\sfrac{1}{2}}^t \int_{\R} \eta^\alpha_n (X^\alpha_n (s, 0)-y) u^\alpha_n (s, y) dy ds
    \\ & \stackrel{\eqref{e:implica}}{=}
      X^\alpha_n (\sfrac{1}{2}, 0)
    + \int_{\sfrac{1}{2}}^t \int_{X^\alpha_n (s, 0)}^{X^\alpha_n (s, \sfrac{1}{2})} \eta^\alpha_n (X^\alpha_n (s, 0)-y) u^\alpha_n (s, y) dy ds.
\end{split}
\end{equation}
We now use~\eqref{e:419}, which implies that for every $y \in [X^\alpha_n (s, 0), X^\alpha_n (s, \sfrac{1}{2})]$, 
    $s \in \R_+$ and $n \ge 18$ 
\begin{equation} \label{e:quasiquasi}
    -3/4 \leq X^\alpha_n (s, 0)-y \leq 0 \stackrel{\eqref{e:etaalpha},\eqref{e:bderre}}{\implies}  \eta^\alpha_n (X^\alpha_n (s, 0)-y) \ge \alpha.
\end{equation}
By plugging the above inequality into~\eqref{e:2D}, we get 
\begin{equation*}
\begin{split}
      X^\alpha_n (t, 0)  & \ge   X^\alpha_n (\sfrac{1}{2}, 0)
    + \int_{\sfrac{1}{2}}^t  \int_{X^\alpha_n (s, 0)}^{X^\alpha_n (s, \sfrac{1}{2})}  \alpha \, u^\alpha_n (s, y) dy ds \\& 
    \stackrel{\eqref{e:pushf2}}{=}  X^\alpha_n (\sfrac{1}{2}, 0)+  \int_{\sfrac{1}{2}}^t \alpha ds =  X^\alpha_n (\sfrac{1}{2}, 0) +\alpha [t -\sfrac{1}{2}].
\end{split}
\end{equation*}
By passing to the $n \to +\infty$ limit in the above inequality and recalling {\sc Step 2} we establish the first inequality in~\eqref{e:inmezzo3}. 
To establish the second inequality in~\eqref{e:inmezzo3} we use an analogous argument, but instead of~\eqref{e:quasiquasi} we use the observation that, for every  
 $y \in [X^\alpha_n (s, 0), X^\alpha_n (s, \sfrac{1}{2})]$, it holds that $X^\alpha_n (s, \sfrac{1}{2})-y \ge 0$ and hence $\eta^\alpha_n (X^\alpha_n (s, 0)-y) \leq \alpha$.
This concludes the proof of Theorem~\ref{t:cex}. \qed

\bibliographystyle{alpha}
\bibliography{bvk}

\end{document}